\documentclass[reqno,a4paper]{amsart}

\usepackage[
  style=numeric-comp,
  sorting=nty,
  sortcites=true,
  doi=false,
  url=false,
  giveninits=true
]{biblatex}
\renewbibmacro{in:}{}
\usepackage{amsmath}
\usepackage{amsthm}
\usepackage{amssymb}

\numberwithin{equation}{section}

\theoremstyle{plain}
\newtheorem{theorem}{Theorem}[section]
\newtheorem{lemma}[theorem]{Lemma}
\newtheorem{proposition}[theorem]{Proposition}
\newtheorem{conjecture}[theorem]{Conjecture}
\theoremstyle{remark}
\newtheorem{remark}[theorem]{Remark}

\newcommand{\Cstar}{\mathbb{C}^{\ast}}
\newcommand{\qPII}{\ensuremath{q\text{-}\mathrm{P}_{\mathrm{II}}}}

\title[Quicksilver Asymptotics and Pole Arrays for \(q\)-Painlevé II] {Quicksilver Asymptotics and Pole Arrays for \(q\)-Painlevé II}
\author{Joshua Holroyd}
\date{\today}

\begin{document}

\begin{abstract}
 We construct local families of solutions of the $q$-difference second Painlev\'e equation that vanish at infinity with a common algebraic expansion, parameterised by the amplitude of an exponentially small correction. Considering an appropriate boundary region, a formal expansion that retains two independently excited exponential perturbations is developed, and subsequent transseries summation predicts the location of simple poles. Adopting a $q$-theta function realisation of the exponential factors, we prove an asymptotic pole-location formula for a one-parameter family of true solutions. These poles approach the aforementioned boundaries where exponential behaviour fundamentally changes, providing an analogue of pole arrays near anti-Stokes directions in classical Painlev\'e asymptotics.
\end{abstract}

\maketitle

\section{Introduction}

A well-studied problem in classical Painlev\'e asymptotics is to relate a solution's expansion in a pole-free sector near infinity to its behaviour near the sector's boundary. Boutroux's analysis of the first and second Painlev\'e equations identified solutions, called \emph{int\'egrales tronqu\'ees}, which admit algebraic asymptotic expansions and are pole-free near infinity in appropriate sectors \cite{Boutroux1913,Boutroux1914}. Such an expansion in a given sector may belong to a family of solutions, as generically a free parameter multiplies an exponentially small correction, invisible at every algebraic order. Analysis of such an exponential perturbation near the boundary of its decay region is essential to understanding where poles appear \cite{CostinCostin2001}.

Joshi's \emph{quicksilver} solutions \cite{Joshi2015} of the $q$-difference first Painlev\'e equation are analogous to Boutroux's tronqu\'ee solutions. These solutions vanish at infinity according to a divergent formal power series; \cite{Joshi2015} proves the existence of true analytic solutions with that expansion on suitable $q$-domains. Since a $q$-difference equation relates iterates along a $q$-spiral, analogous pole-free domains are no longer ordinary sectors of the complex plane.

In this article we study the corresponding vanishing asymptotic behaviour for the $q$-difference second Painlev\'e equation
\begin{equation}
 f(qt)f(t/q) =\frac{a(1+t f(t))} {f(t)\bigl(f(t)+t\bigr)}, \qquad a,q\in\Cstar, \label{eq:qPII}
\end{equation}
where $\Cstar=\mathbb C\setminus\{0\}$. We write $t_n=t_0q^n$ and $f_n=f(t_n)$, with $t_0\in\Cstar$ and $n\in\mathbb Z$.  Throughout, $|q|>1$ and we consider $n\to+\infty$; the case $|q|<1$ is obtained by reversing the iteration. We denote $p=q^2$ throughout this paper. The balance $f_n\sim-1/t_n$ determines the formal algebraic expansion considered in this paper. We call true solutions asymptotic to this expansion quicksilver solutions, as in \cite{Joshi2015}.

In the differential setting, an exponential correction will change from decay to oscillation or growth as the direction of approach to infinity varies. We use \emph{anti-Stokes} for directions along which the exponential contribution is purely oscillatory \cite{CostinCostin2001}. In such directions, contributions that are negligible in the decay region become significant. Here, one may develop asymptotic expansions featuring two independent exponential amplitudes, as the two modes no longer exhibit exponential growth or decay. Such two-parameter families are well established in the differential setting; see, for instance, \cite{SchiappaVaz2014,VanSpaendonckVonk2022}.

For \qPII{}, linearisation about the vanishing balance gives two independent perturbations, one growing and one decaying along a forward $q$-orbit. To investigate behaviour in which this separation by exponential size is lost, we retain both modes in a formal expansion. We call this expansion \emph{bimodal}; when one of the free exponential amplitudes is zero, we call it \emph{one-sided}. The nonlinear interaction of these independently excited modes modifies both the algebraic background and the exponential phase seen in the one-sided case.

Costin and Costin \cite{CostinCostin2001} develop expansions that remain valid as an initially small exponential correction becomes finite near anti-Stokes boundaries. They group powers of that exponential into coefficient functions and relate the singularities of these functions to those of the solution. For the first and second Painlev\'e equations, this gives pole arrays near anti-Stokes directions, with locations determined by the free parameter in the exponentially small correction. This reorganisation is called \emph{transasymptotic} analysis; explicit singularity-location calculations for the first Painlev\'e equation (and its perturbed form) are given in \cite{OldeDaalhuis2022}.

For $q$-PI, \cite{JoshiLustriLuu2019} considers the switching of exponentially small corrections across Stokes curves in a continuum-limit scaling with $q\to1$ and $n\to\infty$. For fixed $q$, \cite{JoshiOldeDaalhuis2025} develops exponentially improved asymptotics and Stokes analysis for $q$-PI. Their one-sided transseries summation provides a rational expression and predicts poles from its denominator. In this article, we extend these ideas by first considering a bimodal transseries in the \qPII{} case and, under explicit local hypotheses, prove the existence and locations of simple poles for the one-parameter family of true solutions.

For background on meromorphic $q$-difference asymptotics and the role of $r$-theta functions, including a discussion of nonlinear $r$-theta-transseries, see Zhang \cite{Zhang2006}. In Section \ref{sec:meromorphic-realisation} we implement a fixed $r$-theta realisation of the exponential factors, leading to an explicit geometric analogue of classical anti-Stokes boundaries, which, as in the differential setting, are approached by the poles of true meromorphic solutions.

\subsection{Main results}

Theorem~\ref{thm:one-sided-existence} constructs a local one-parameter family of quicksilver solutions sharing the algebraic expansion with leading term $-1/t$ on neighbourhoods of forward $q$-orbits.

For the formal bimodal expansion, Proposition~\ref{prop:phase-selection} determines the amplitude-dependent exponential phase and leading (geometric) coefficients. Under the proposed correction formulas of Conjecture~\ref{conj:polynomial-corrections}, rational summation indicates the positions of simple poles (Proposition~\ref{prop:formal-pole-displacement}).

For a fixed $r$-theta function realisation, Proposition~\ref{prop:one-sided-poles} establishes a local array of simple poles for a normalised one-sided solution and gives its first displacement correction. These poles approach convergence boundaries of the corresponding geometric sums of Section~\ref {sec:formal-bimodal}.

\subsection{Outline of the paper}

Section~\ref{sec:quicksilver} establishes the quicksilver analytic existence result by summation and contraction. In Section~\ref{sec:formal-bimodal} we develop the bimodal expansion and its formal pole predictions (with further coefficient formulas supplied in Appendix~\ref{app:Coefficients}). Section~\ref{sec:meromorphic-realisation} proves pole estimates for the one-sided solutions and discusses the corresponding anti-Stokes geometry (with the core uniform estimate proof given in Appendix~\ref{app:one-sided-estimates}).

\section{Existence of quicksilver solutions}
\label{sec:quicksilver}

Following the Borel--Ritt and summation approach of \cite{Joshi2015} regarding $q$-PI, we construct a local one-parameter family of solutions of \eqref{eq:qPII} sharing a vanishing algebraic expansion near infinity. After realising the formal series background, we construct exponential modes that satisfy a linear perturbation equation and a corresponding summation operator, and then solve the nonlinear perturbation equation by contraction. Estimates are uniform on the local $q$-domains defined below.

\subsection{Algebraic background and local domains}

Substitution into \eqref{eq:qPII} determines the unique formal expansion with leading term $-1/t$ as $|t|\to\infty$:
\begin{equation}
 \hat D(t):=\sum_{j\geq0}\hat d_j/t^{2j+1} =-\frac1t-\frac{1}{at^3}+\mathcal O\left(1/t^{5}\right). \label{eq:algebraic-series}
\end{equation}
The coefficient recurrence and its growth estimates are given in \cite[\S3.2.1 and Appendix A]{Holroyd2025}; only the existence of this formal series is needed here.

Choose an open base neighbourhood $U\Subset\Cstar$ sufficiently small so that the sets (\emph{dilates}) $q^nU$, $n\in\mathbb Z$, are pairwise disjoint, and for $N\in\mathbb{Z}$ define
\[
 \Omega_N:=\bigcup_{n\ge N}q^nU,\qquad (t_n=q^nt_0,\quad t_0\in U),
\]
noting then that each point of $\Omega_N$ corresponds to a unique base point of $U$.  A finite collection of sectors can be chosen so that each $q^nU$ lies in a closed subsector of one of them; the Borel--Ritt theorem \cite{Friedrichs1953,HsiehSibuya1999} supplies a realisation of \eqref{eq:algebraic-series} on each sector. Selecting one such realisation on each disjoint dilate defines a holomorphic function $D$ on a sufficiently late $\Omega_N$, with $D\sim\hat D$ uniformly there. We replace $U$ with a fixed, relatively compact subdomain (still denoted $U$); all estimates below are uniform for $t_0\in\overline U$.

Denote $D_n=D(t_n)$ and increase $N$ so that $D_n(D_n+t_n)\ne0$. As the series \eqref{eq:algebraic-series} solves \eqref{eq:qPII} formally, the residual defined
\begin{equation}
 \mathcal E_n:=D_{n+1}D_{n-1} -\frac{a(1+t_nD_n)}{D_n(D_n+t_n)} \label{eq:algebraic-residual}
\end{equation}
is algebraically flat: for every $M\ge0$, $\mathcal E_n=\mathcal O(1/t_n^{M})$ uniformly in the base point. All shifted equations below are imposed at $n\ge N+1$, where the three adjacent values belong to $\Omega_N$.

\subsection{Perturbation equation and estimates}

To obtain an exact solution near the background $D$, set $f_n=D_n+v_n$ and write the forward mapping as
\[
 F_n(x,y)=\frac{a(1+t_nx)}{yx(x+t_n)}.
\]
We choose the coefficients of the linear operator as
\[
 \beta_n= \frac{D_{n-1}D_{n+1}(2D_n+t_n)-at_n} {D_{n-1}D_n(D_n+t_n)},\qquad \gamma_n=\frac{D_{n+1}}{D_{n-1}},
\]
so that \eqref{eq:qPII} is equivalent to
\begin{equation}
 (\mathcal Av)_n:=v_{n+1}+\beta_nv_n+\gamma_nv_{n-1} =\mathcal N_n(v_n,v_{n-1}), \label{eq:nonlinear-perturbation}
\end{equation}
for exact remainder $\mathcal N_n$ that is by definition
\[
 \mathcal N_n(x,y)=F_n(D_n+x,D_{n-1}+y)-D_{n+1}+\beta_nx+\gamma_ny.
\]
The following partial derivatives of $F_n$ are evaluated at $(D_n,D_{n-1})$.  Direct differentiation gives $\mathcal N_n(0,0)=-\mathcal E_n/D_{n-1}$ and
\begin{align*}
 \beta_n+\partial_xF_n=\frac{(2D_n+t_n)\mathcal E_n}{D_{n-1}D_n(D_n+t_n)},\quad\gamma_n+\partial_yF_n=\mathcal E_n/D_{n-1}^2.
\end{align*}
So, the constant and linear terms of $\mathcal N_n$ are flat, and with $\Delta X_n:=X_n-X_{n-1}$, the choice of $\gamma_n$ also gives
\begin{equation}
 \frac{P_n(\mathcal Av)_n-v_n(\mathcal AP)_n}{D_{n+1}D_n} =\Delta\left\{ \frac{P_nv_{n+1}-P_{n+1}v_n}{D_{n+1}D_n}\right\}, \label{eq:ExactDiff}
\end{equation}
for any sequences $P_n,v_n$.  This follows by cancelling the terms containing $\beta_n$ and using $\gamma_n=D_{n+1}/D_{n-1}$.

\begin{lemma}[Holomorphy and nonlinear estimates]
\label{lem:nonlinear-estimates}
Fix $0<\delta<1$. For sufficiently large $N$ and $n\ge N+1$, $\mathcal N_n(x,y)$ is holomorphic in a neighbourhood of the closed polydisc
\begin{equation}
 |x|\le\delta|D_n|,\qquad |y|\le\delta|D_{n-1}|. \label{eq:shrinking-polydisc}
\end{equation}
There is a constant $C>0$ and, for each $M\ge0$, a number $\eta\ge0$ such that
\begin{align}
 |\mathcal N_n(x,y)| &\le C|t_n|^3(|x|+|y|)^2+\eta/|t_n|^{M}, \label{eq:nonlinear-estimate}\\
 \|\nabla\mathcal N_n(x,y)\| &\le C|t_n|^3(|x|+|y|)+\eta/|t_n|^{M}. \label{eq:nonlinear-lipschitz}
\end{align}
Besides the fixed data, $C$ may depend on $\delta$, and $\eta$ may also depend on $M,N$, with $\eta\to0$ as $N\to\infty$ for each fixed $M$. Both are independent of $n$ and $t_0\in\overline U$.
\end{lemma}

\begin{proof}
On a slightly larger relative polydisc, still of relative radius less than one, the three denominator factors in $F_n$ satisfy
\[
 |D_n+x|\asymp|D_{n-1}+y|\asymp1/|t_n|,\qquad |D_n+t_n+x|\asymp|t_n|.
\]
They are nonzero for large $N$, and $|F_n|\le C|t_n|$ there. Cauchy's estimate on radii proportional to $1/|t_n|$ gives $\|\nabla^2F_n\|\le C|t_n|^3$ on \eqref{eq:shrinking-polydisc}.  Taylor's formula therefore bounds the quadratic remainder and its gradient by the first terms in \eqref{eq:nonlinear-estimate}--\eqref{eq:nonlinear-lipschitz}. The affine part of $\mathcal N_n$ and its gradient are bounded by $C|t_n||\mathcal E_n|$ and $C|t_n|^2|\mathcal E_n|$, respectively.  For example, after increasing $N$ so that $|t_n|\ge1$, one may take
\[
 \eta=C \sup_{\Lambda} |t_n|^{M+2}|\mathcal E_n|,\qquad \Lambda=\left\{(t_0,n)\in\overline U\times\mathbb{Z}\,:\,n\ge N+1\right\}.
\]
Uniform flatness of \eqref{eq:algebraic-residual} makes this tend to zero.
\end{proof}

\subsection{Linear modes and the summation operator}

We construct a recessive solution $P^-$ and a dominant solution $P^+$ of the linear, homogeneous equation $\mathcal AP=0$, then use them to define a right inverse of $\mathcal A$ by summation.

\begin{lemma}[Linear dichotomy]
\label{lem:linear-dichotomy}
On a sufficiently late orbit tail, $\mathcal AP=0$ has a fundamental pair
\begin{align}
 P_n^-&=a^{-n}t_0^{-2n}q^{-n^2-2n} \{1+\mathcal O(1/t_n^{2})\}, \label{eq:Pminus}\\
 P_n^+&=a^nt_0^{2n}q^{n^2-2n} \{1+\mathcal O(1/t_n^{2})\}. \label{eq:Pplus}
\end{align}
The estimates are uniform for $t_0\in\overline U$, and the solutions are holomorphic in the base point. Both are nonzero on a sufficiently late tail; $P^-$ decays faster than every algebraic power of $1/|t_n|$, whereas $P^+$ grows faster than every algebraic power of $|t_n|$.
\end{lemma}

\begin{proof}
The background expansion $D$ gives
\[
 \beta_n=-(a/q) t_n^2+\mathcal O(1),\qquad \gamma_n=1/q^2+\mathcal O(1/t_n^{2}).
\]
For a preliminary dominant solution, choose $P_N=a^Nt_0^{2N}q^{N^2-2N}$ at a sufficiently large $N$, and set $R_n=P_{n+1}/P_n$.  Starting with $R_N=at_N^2/q$, the recurrence
\[
 R_n=-\beta_n-\gamma_n/R_{n-1},\qquad n\ge N+1,
\]
preserves $R_n=at_n^2/q+\mathcal O(1)$ uniformly: a bound of this form at $n-1$ makes $1/R_{n-1}=\mathcal O(1/t_{n}^{2})$, and the displayed bounds on $\beta_n,\gamma_n$ give the next step.  A fixed error bound larger than the bounded part of $-\beta_n$ is preserved after increasing $N$. In particular, no denominator vanishes, and the ratios are holomorphic in $t_0$. The product of their relative corrections converges to a nonzero holomorphic factor, since $\sum_{n\ge N}1/|t_n|^{2}<\infty$. Dividing by that factor normalises the solution to \eqref{eq:Pplus}; the remaining tail of the product is $1+\mathcal O(1/t_n^{2})$.

Reduction of order now gives the recessive solution directly:
\[
 P_n^-=at_0^4P_n^+\sum_{j=n}^{\infty} \frac{D_{j+1}D_j}{P_{j+1}^+P_j^+}.
\]
Successive summands above have ratio $\mathcal O(1/t_j^{4})$, so the sum is its first term times $1+\mathcal O(1/t_n^{4})$. Together with \eqref{eq:Pplus} and the background expansion \eqref{eq:algebraic-series}, this yields \eqref{eq:Pminus}. Moreover,
\[
 P_n^+P_{n+1}^--P_{n+1}^+P_n^-=-at_0^4D_{n+1}D_n\ne0.
\]
Identity \eqref{eq:ExactDiff} proves that $\mathcal AP^-=0$, and the nonzero determinant proves independence. Uniform convergence of the series proves holomorphy in the base point.
\end{proof}

Fix such a pair \eqref{eq:Pminus}--\eqref{eq:Pplus} and restrict it to tails starting at $N$. For a forcing sequence $h=(h_n)_{n\ge N}$, define for $n\ge N$
\begin{equation}
 (\mathcal Kh)_n=-P_n^+\sum_{j=n}^{\infty} \frac{D_{j+1}D_j}{P_{j+1}^+P_j^+} \sum_{l=N+1}^{j}\frac{P_l^+h_l}{D_{l+1}D_l}, \label{eq:nested-summation-operator}
\end{equation}
with an empty inner sum interpreted as zero. The finite lower limit in the inner sum allows algebraically decaying forcing; convergence and bounds are proved below. Whenever the series converges absolutely, subtraction at adjacent indices gives
\[
 \frac{P_n^+(\mathcal Kh)_{n+1}-P_{n+1}^+(\mathcal Kh)_n} {D_{n+1}D_n} =\sum_{l=N+1}^{n}\frac{P_l^+h_l}{D_{l+1}D_l}.
\]
Taking its difference and using \eqref{eq:ExactDiff} proves
\begin{equation}
 (\mathcal A\mathcal Kh)_n=h_n,\qquad n\ge N+1. \label{eq:right-inverse}
\end{equation}
The upper limit at infinity in \eqref{eq:nested-summation-operator} gives $(\mathcal Kh)_n/P_n^+\to0$ and removes the free dominant homogeneous component. This is the nested summation construction used in \cite[\S3 and Appendix B]{Joshi2015}.

\begin{lemma}[Bounds for the summation operator]
\label{lem:summation-bound}
For $r\ge0$, let $X_r$ be the Banach space of sequences on $n\ge N$ with norm $\|h\|_r=\sup_{n\ge N}|t_n|^r|h_n|$. For each $r$ there are constants $N_r,C_r$ such that, uniformly for $N\ge N_r$ and $t_0\in\overline U$,
\begin{equation}
 \|\mathcal Kh\|_{r+2}\le C_r\|h\|_r. \label{eq:weighted-kernel-bound}
\end{equation}
Consequently the operator norm satisfies $\|\mathcal K\|_{X_r\to X_r}\le C_r/|t_N|^{2}$. For any fixed earlier starting index on which the pair $P^\pm$ is nonzero, $\mathcal K:X_r\to X_{r+2}$ remains bounded, with a constant allowed to depend on that starting index.
\end{lemma}

\begin{proof}
Uniformly on late tails, $|D_{j+1}D_j|\asymp1/|t_j|^{2}$ and $|P_j^+/P_{j+1}^+|\le C/|t_j|^{2}$. For $h\in X_r$, the inner summands in \eqref{eq:nested-summation-operator} are bounded by $C\|h\|_r|P_l^+||t_l|^{2-r}$. The ratio of these majorants in the decreasing direction of $l$ is at most $C_r/|t_l|^{2}$, hence at most $1/2$ after increasing $N_r$. The inner sum is therefore bounded by a constant multiple of its upper endpoint, giving
\[
 |(\mathcal Kh)_n| \le C_r\|h\|_r|P_n^+| \sum_{j=n}^{\infty}\frac{1}{|P_{j+1}^+||t_j|^{r}}.
\]
Successive terms of this last series also have a ratio at most $C_r/|t_{j+1}|^{2}\le1/2$. Bounding it by its first term gives $|(\mathcal Kh)_n|\le C_r\|h\|_r/|t_n|^{r+2}$ and proves absolute convergence.  For a fixed earlier starting index, split the inner sum at $N_r$. Its finite initial part contributes a multiple of the recessive tail in the proof of Lemma~\ref{lem:linear-dichotomy}; the finitely many initial output values are bounded separately. This proves the final assertion.
\end{proof}

\subsection{Construction of the one-sided family}

Nonlinear estimates of Lemma~\ref{lem:nonlinear-estimates} and the summation bound of Lemma~\ref{lem:summation-bound} supply a contraction with a free coefficient multiplying the recessive mode. For $v=(v_n)_{n\ge N}$, set boundary $(\mathcal Nv)_N=0$ and then $(\mathcal Nv)_n=\mathcal N_n(v_n,v_{n-1})$ for $n\ge N+1$.

\begin{theorem}[One-sided quicksilver family]
\label{thm:one-sided-existence}
Fix the base domain $U$ above and a nonempty compact set $K\subset\mathbb C$. There are $N_*$ and $\rho>0$ such that, for every $(k,t_0)\in K\times\overline U$, the summation equation
\begin{equation}
 f_n\equiv f_n(k)=D_n+v_n(k),\qquad v(k)=kP^-+\mathcal K\mathcal Nv(k) \label{eq:one-sided-fixed-point}
\end{equation}
has a unique solution in $\{v\in X_4:\|v\|_4\le\rho\}$, with starting index $N=N_*$ in $\mathcal K$.  It solves \eqref{eq:qPII} at $n\ge N_*+1$ and satisfies $v(k)\in\bigcap_{r\ge0}X_r$ uniformly on these parameter sets. Thus $f_n\sim D_n$ to all algebraic orders. Each $f_n$ is holomorphic in $(k,t_0)$ near $K\times\overline U$.  For fixed $k$, the disjoint dilates $q^nU$ therefore define a holomorphic solution on $\Omega_{N_*}$, with the equation imposed off its initial layer.
\end{theorem}

\begin{proof}
Fix $0<\delta<1$ in Lemma~\ref{lem:nonlinear-estimates} and $\rho>0$. For sufficiently large $N$, every $v\in X_4$ with $\|v\|_4\le\rho$ has $(v_n,v_{n-1})$ in the relative polydisc \eqref{eq:shrinking-polydisc} for $n\ge N+1$, because $|D_n|\asymp1/|t_n|$. For each $(k,t_0)\in K\times\overline U$, define $\mathcal T_kv=kP^-+\mathcal K\mathcal Nv$ on this ball. For any $v,w$ in the ball, Lemmas~\ref{lem:nonlinear-estimates} and \ref{lem:summation-bound} give
\begin{align*}
 \|\mathcal T_kv-kP^-\|_4 &\le C|t_N|^{-2}(\eta+\rho^2/|t_N|),\\
 \|\mathcal T_kv-\mathcal T_kw\|_4 &\le C|t_N|^{-2}(\eta+\rho/|t_N|) \|v-w\|_4.
\end{align*}
Here fixed powers of $|q|$ are absorbed into $C$, and $\eta$ is a common bound for the choices $M=0,4$ in Lemma~\ref{lem:nonlinear-estimates}. Since $\|P^-\|_4\to0$, $\eta\to0$ and $\inf_{t_0\in\overline U}|t_N|\to\infty$, a single $N_\ast$ makes $\mathcal T_k$ a contraction of this ball into itself, uniformly for $(k,t_0)\in K\times\overline U$. Its fixed point solves \eqref{eq:nonlinear-perturbation}, and hence \eqref{eq:qPII}, by \eqref{eq:right-inverse} and $\mathcal AP^-=0$.

For this fixed starting index, write $v=v(k)$ for the fixed point. If $v\in X_r$ with $r\ge4$, then \eqref{eq:nonlinear-estimate} gives $\mathcal Nv\in X_{2r-3}$. The gain of two powers in Lemma~\ref{lem:summation-bound} and the flatness of $P^-$ imply $v\in X_{2r-1}$. Iterating $r\mapsto2r-1$ from $r=4$ proves membership in every $X_r$, with uniform bounds on the stated parameter sets. The estimates can be made on slightly larger compact parameter sets. The Picard iterates and the summation kernels are holomorphic, and their locally uniform convergence proves the claimed holomorphic dependence. The disjointness of the dilates then gives the asserted function on the $q$-domain.
\end{proof}

\begin{remark}[Normalisation of the parameter]
For the fixed starting index $N=N_\ast$, the summation operator has $P_N^+(\mathcal Kh)_{N+1}-P_{N+1}^+(\mathcal Kh)_N=0$. Consequently, the parameter in \eqref{eq:one-sided-fixed-point} is precisely
\[
 k=\frac{P_N^+v_{N+1}(k)-P_{N+1}^+v_N(k)} {P_N^+P_{N+1}^--P_{N+1}^+P_N^-}.
\]
The denominator is nonzero by Lemma~\ref{lem:linear-dichotomy}, so distinct parameters give distinct solutions. This fixes an orbitwise integration constant at the starting index.
\end{remark}

\section{Formal bimodal expansion}
\label{sec:formal-bimodal}

We develop a two-parameter family of formal solutions of \eqref{eq:qPII} near infinity retaining both exponential modes suggested by \eqref{eq:Pminus}--\eqref{eq:Pplus}. Assuming a region where these exponential modes abandon their dominant and recessive roles, their nonlinear interaction according to \eqref{eq:qPII} perturbs the background series $D$ and the phase seen in \eqref{eq:Pminus}--\eqref{eq:Pplus}. Under our proposed coefficient formulas, transasymptotic summation then gives functions consistent with an expansion about a simple pole.

In Section~\ref{sec:meromorphic-realisation}, we fix an $r$-theta function realisation of the exponential modes, providing a concrete interpretation of the appropriate anti-Stokes-like boundary regions, and we proceed to establish pole-location estimates for a one-sided family of true solutions.

\subsection{Preliminaries}

We introduce the exponential mode
\begin{equation}\label{eq:QExpFactor}
 Z:=Z_n=(at_0^2\sigma)^nq^{n^2}\qquad\text{satisfying}\qquad Z_{n+1}=qa\sigma t_n^2Z_n.
\end{equation}
The constant $\sigma\ne0$ is a \emph{nonlinear phase} to be determined by \qPII{}. Throughout the remainder of the paper, we make use of the transformed function $u(t)=-tf(t)$, further writing $u_n=u(t_n)=-t_nf_n$. The proposed bimodal formal expansion as $|t|\to\infty$ is
\begin{equation}
 u_n=-\sum_{m\in\mathbb Z}\sum_{j\geq0} c_{m,j}\epsilon^{3|m|+2j-2}Z^m,\qquad \epsilon:=1/t_n, \quad c_{0,0}=0. \label{eq:u-bimodal-series}
\end{equation}
We call $m$ the \emph{harmonic number} and $j$ the \emph{algebraic correction grade}. The choice $c_{0,0}=0$ selects the (zero-harmonic) vanishing background series $f=\mathcal{O}(1/t)$; this differs from $D$, depending on the two free parameters. For coefficient matching, $Z$ is treated as a formal Laurent variable.

The two free parameters (\emph{amplitudes}) are $k_1:=c_{1,0}$ and $k_2:=c_{-1,0}$, and we denote their scaled product $\kappa:=ak_1k_2$. We utilise the assumption of regular coefficientwise limits as either amplitude vanishes with the other held fixed, such that the one-sided expansions of Section \ref{sec:quicksilver} are reproduced.

After setting
\begin{equation}
 X_+=\epsilon^3Z,\qquad\text{and}\qquad X_-=\epsilon^3/Z, \label{eq:bimodal-scales}
\end{equation}
grouping \eqref{eq:u-bimodal-series} by algebraic grade gives
\begin{equation}
 u_n=-\sum_{j\geq1}d_j\epsilon^{2j-2} -\sum_{m\geq1}\sum_{j\geq0}\epsilon^{2j-2} \left(c_{m,j}X_+^m+c_{-m,j}X_-^m\right),\quad d_j:=c_{0,j}. \label{eq:u-bigraded-series}
\end{equation}
The positive and negative harmonic tails are related by the formal reflection
\[
 \mathcal R:\qquad X_+\leftrightarrow X_-,\quad m\mapsto-m,\quad p\mapsto 1/p,\quad k_1\leftrightarrow k_2.
\]
With $a$, $\kappa$, $\sigma$, and $\epsilon$ fixed, $\mathcal R$ interchanges the positive and negative harmonic equations. We therefore write $c_{-m,j}=\mathcal R(c_{m,j})$ and derive only the positive equations. The zero-harmonic formulas are invariant under $\mathcal R$.

\subsection{Leading coefficients and phase selection}

We first solve the leading-coefficient equations for fixed $\sigma$, then select the phase by requiring regular limits as either of the amplitudes vanish.

\begin{proposition}[Leading harmonic coefficients]
\label{prop:leading-harmonic-coefficients}
Suppose $\kappa\sigma\ne0$ and $\kappa\sigma^2\ne1$.  The leading coefficient equations of \eqref{eq:u-bigraded-series} determine
\begin{equation}
 d_1=1/(\kappa\sigma^2-1). \label{eq:leading-zero-mode}
\end{equation}
If $d_1^2\ne\kappa\sigma$, the leading harmonics are given, for $m\geq1$, by
\[
 c_{m,0}=k_1^mp^{m-1}a^{1-m}\hat c_m\qquad\text{and}\qquad c_{-m,0}=k_2^mp^{1-m}a^{1-m}\hat c_m,
\]
where
\begin{equation}
 \begin{aligned}
 \hat c_m={}&\frac{d_1^2\sigma-1}{\sigma(d_1^2-\kappa\sigma)} \left(-\frac{d_1}{\kappa\sigma^2}\right)^{m-1}-\frac{1}{d_1\sigma(d_1^2-\kappa\sigma)} \left(-\frac{1}{d_1\sigma}\right)^{m-1}.
 \end{aligned}
 \label{eq:leading-tail-decomposition}
\end{equation}
In particular, $\hat c_1=1$, and the normalised sequence $\hat c_m$ is invariant under $\mathcal R$.
\end{proposition}

\begin{proof}
The prescribed shifts give
\begin{align*}
 u_{n+1} &=-\sum_{m\in\mathbb Z}\sum_{j\geq0} c_{m,j}q^{m-3|m|-2j+2}(a\sigma)^m \epsilon^{3|m|-2m+2j-2}Z^m,\\
 u_{n-1} &=-\sum_{m\in\mathbb Z}\sum_{j\geq0} c_{m,j}q^{m+3|m|+2j-2}(a\sigma)^{-m} \epsilon^{3|m|+2m+2j-2}Z^m.
\end{align*}
Both shifted series have powers of $\epsilon$ bounded below by $-1$, and only finitely many pairs $(m,j)$ contribute to each fixed power. Substituting $f_n=-u_n/t_n$ in \eqref{eq:qPII} gives
\begin{equation}
 \epsilon^2u_{n+1}u_nu_{n-1}(1-\epsilon^2u_n)=a(u_n-1). \label{eq:scaled-bimodal-equation}
\end{equation}
The $\epsilon^0Z^0$ equation is $a(1+d_1)=a\kappa\sigma^2d_1$, which gives \eqref{eq:leading-zero-mode}.

For every positive harmonic $m\geq2$, only the term $\epsilon^2u_{n+1}u_nu_{n-1}$ contributes at the leading order $\epsilon^m$.  Equating its coefficient to zero gives
\begin{align*}
 0={}&q^{2-2m}(a\sigma)^m d_1^2c_{m,0} +q^{4-2m}(a\sigma)^{m-1}k_1d_1c_{m-1,0}\\
 &+q^{-2m}(a\sigma)^{m+2}k_2d_1c_{m+1,0} +q^{2-2m}(a\sigma)^{m+1}k_1k_2c_{m,0}.
\end{align*}
Therefore, for all $m\geq2$, using $p=q^2$ and $\kappa=ak_1k_2$ yields
\begin{equation}
 k_2(a\sigma)^2c_{m+1,0} +p\left(d_1+\kappa\sigma/d_1\right)c_{m,0} +k_1p^2/(a\sigma)c_{m-1,0}=0. \label{eq:leading-c-recurrence}
\end{equation}
Meanwhile, for $m=1$, the term $au_n$ on the right of \eqref{eq:scaled-bimodal-equation} contributes at the same order, giving
\begin{equation}
 p/(d_1\sigma)c_{1,0} =k_2(a\sigma)^2c_{2,0} +p\left(d_1+\kappa\sigma/d_1\right)c_{1,0}. \label{eq:leading-c-boundary}
\end{equation}

Under the normalisation in the statement, \eqref{eq:leading-c-recurrence} becomes
\begin{equation}
 \hat c_{m-1}+\sigma(d_1+\kappa\sigma/d_1)\hat c_m +\kappa\sigma^3 \hat c_{m+1}=0, \qquad m\geq2. \label{eq:leading-left-boundary}
\end{equation}
The initial data are $\hat c_1=1$ and, by \eqref{eq:leading-c-boundary}, $\hat c_2=-(1+d_1^3\sigma)/(d_1^2\kappa\sigma^3)$. The characteristic polynomial factors as
\[
 1+\sigma(d_1+\kappa\sigma/d_1)\lambda+\kappa\sigma^3\lambda^2 =(1+d_1\sigma\lambda) (1+(\kappa\sigma^2/d_1)\lambda).
\]
Its two roots, $-1/(d_1\sigma)$ and $-d_1/(\kappa\sigma^2)$, are distinct under the stated hypothesis.  Solving for their coefficients using $\hat c_1$ and $\hat c_2$ gives \eqref{eq:leading-tail-decomposition}. Reflection by $\mathcal R$ leaves this recurrence and its initial data unchanged, so the negative harmonics have the same normalised sequence.
\end{proof}

\begin{proposition}[Phase selection]
\label{prop:phase-selection}
Assume that $\sigma\equiv\sigma(\kappa)$ is analytic, or a formal power series, at $\kappa=0$, with $\sigma(0)=1$.  The leading coefficients have regular one-sided limits as either amplitude tends to zero, with the other held fixed (coefficientwise in the formal case), if and only if $d_1^2\sigma=1$, equivalently
\begin{equation}
 \sigma=(\kappa\sigma^2-1)^2. \label{eq:phase-relation}
\end{equation}
On the branch through $\sigma(0)=1$, the leading coefficients in the normalisation above reduce to
\begin{equation}
 \hat c_m=(-d_1)^{m-1},\qquad m\geq1. \label{eq:leading-coefficients}
\end{equation}
The phase is uniquely determined near $\kappa=0$ and depends on the two amplitudes only through $\kappa=ak_1k_2$.
\end{proposition}

\begin{proof}
Since $d_1(0)=-1$ and $\sigma(0)=1$, only the first term in \eqref{eq:leading-tail-decomposition} can be singular at $\kappa=0$. With $k_1\ne0$ fixed, regularity of $c_{m,0}$ requires regularity of $\hat c_m$ for every $m$. This forces $d_1^2\sigma=1$: any finite-order vanishing of its numerator would still produce negative powers of $\kappa$ for sufficiently large $m$. Conversely, substituting this identity into \eqref{eq:leading-tail-decomposition} gives \eqref{eq:leading-coefficients}, which is regular at both endpoints. By \eqref{eq:leading-zero-mode}, the condition is equivalent to \eqref{eq:phase-relation}; the implicit-function theorem gives its unique analytic branch through $(\kappa,\sigma)=(0,1)$, and coefficient comparison gives the unique formal branch.
\end{proof}

At either one-sided endpoint, $\kappa=0$, the selected branch gives $\sigma=1$ and $d_1=-1$.  Equations~\eqref{eq:Pminus}--\eqref{eq:Pplus} and \eqref{eq:bimodal-scales} then give $\epsilon P_n^\pm=t_0^2X_\pm(1+\mathcal O(\epsilon^2))$. Thus, at the level of leading formal harmonics in $f_n=-u_n/t_n$, the endpoints $k_1=0$ and $k_2=0$ recover $(k_2/t_0^2)P_n^-$ and $(k_1/t_0^2)P_n^+$, respectively.  The first is the forward recessive mode used in Theorem~\ref{thm:one-sided-existence}; the second is its reflected counterpart under reversed dominance.

\subsection{Algebraic corrections}

At a fixed algebraic grade $J>0$, extract the geometric factor $c_{m,j}=(-pk_1d_1/a)^mh_{m,j}$. For sufficiently large $m$, the order-$J$ coefficient equations take the form
\[
 d_1h_{m-1}-(1+2d_1)h_m+(1+d_1)h_{m+1}=\mathcal{F}_J,
\]
where $h_m=h_{m,J}$. The forcing $\mathcal F_J$ contains lower-grade coefficients together with low-harmonic and zero-mode data from the current-order equations. As in Proposition~\ref{prop:phase-selection}, analytic or formal dependence on $\kappa$ and regular one-sided limits exclude the singular homogeneous branch, leaving the root $1$. If $\mathcal F_J$ is polynomial in $m$ of degree at most $J-1$, the displayed operator admits a polynomial particular solution of degree at most $J$. This motivates the following conjecture.

\begin{conjecture}[Polynomial-geometric coefficients]
\label{conj:polynomial-corrections}
For every algebraic grade $j\geq0$, there are an integer $M_j\geq1$ and coefficients $\gamma_{k,j}$, $0\le k\le j$, independent of $m$, such that the formal solution selected by regular one-sided coalescence satisfies
\[
 c_{m,j}=\left(-pk_1d_1/a\right)^m \sum_{k=0}^{j}\gamma_{k,j}m^k, \qquad \forall m\geq M_j=j+1.
\]
The negative harmonics are obtained by $c_{-m,j}=\mathcal R(c_{m,j})$.
\end{conjecture}

The leading case $j=0$ follows from Proposition~\ref{prop:phase-selection}; the proposed formulas for $j=1,2$ are verified through finitely many harmonics. The corresponding coefficients are supplied in Appendix~\ref{app:Coefficients}.

\subsection{Geometric sums and formal poles}
\label{subsec:harmonic-summation}

Under the proposed coefficient formulas, summing over harmonics at each algebraic grade gives a rational profile in the exponential mode $Z$. We carry this out at grades $j=0,1,2$ to obtain formal pole-location predictions.

First we set
\[
 x_+=-\frac{pk_1d_1}{a}X_+,\qquad x_-=-\frac{k_2d_1}{ap}X_-=\mathcal R(x_+),
\]
and define the tail profiles by
\[
 H_j=\sum_{m\geq1}c_{m,j}X_+^m,\qquad \mathcal RH_j=\sum_{m\geq1}c_{-m,j}X_-^m.
\]
Thus, we may reorganise \eqref{eq:u-bigraded-series} as
\begin{align}
 u={}&-d_1-d_2\epsilon^2-d_3\epsilon^4+\cdots-(1/\epsilon^2)(1+\mathcal R)\{H_0+H_1\epsilon^2+\epsilon^4H_2+\cdots\}, \label{eq:resummed-grades}
\end{align}
and we compute the following positive-harmonic sums
\begin{align}
 H_0=&\frac{\gamma_{0,0}x_+}{1-x_+},\qquad H_1=c_{1,1}X_+ +\frac{\gamma_{0,1}x_+^2}{1-x_+} +\gamma_{1,1}\frac{x_+^2(2-x_+)}{(1-x_+)^2},\notag\\
 H_2=&c_{1,2}X_++c_{2,2}X_+^2 +\frac{\gamma_{0,2}x_+^3}{1-x_+} +\gamma_{1,2}\frac{x_+^3(3-2x_+)}{(1-x_+)^2}+\gamma_{2,2} \frac{x_+^3(9-11x_++4x_+^2)}{(1-x_+)^3}. \label{eq:resummed-second-grade}
\end{align}
The sums follow from the identity
\[
 \sum_{m\geq1}m^k x^m =\left(x\frac{d}{dx}\right)^k\frac{x}{1-x},\qquad |x|<1,\qquad k\ge0.
\]

Under these coefficient formulas, each displayed harmonic series converges absolutely and uniformly for $|x_\pm|\le\delta<1$. The rational sums continue beyond these discs, with poles at $x_\pm=1$. Since \(x_+x_-=(\kappa d_1^2/a^3)\epsilon^6\), the opposite tail is regular near \(x_\pm=1\) for small \(\epsilon\) and does not contribute to the principal parts used below.

The apparent double and triple pole terms seen in \eqref{eq:resummed-second-grade} are consistent with an expansion about a simple pole whose position depends on $\epsilon$. The following inversion of expansion \eqref{eq:resummed-grades} is the same formal pole-location mechanism used in classical transasymptotic analysis \cite[\S2]{OldeDaalhuis2022}; its reciprocal formulation is described in \cite[\S6.6a]{Costin2009}.

\begin{proposition}[Formal pole locations]
\label{prop:formal-pole-displacement}
Assume the polynomial tails of Conjecture~\ref{conj:polynomial-corrections}, with coefficients \eqref{eq:first-harmonic-correction}--\eqref{eq:second-harmonic-correction} for all $m\ge2$ and $m\ge3$, respectively. Through these two correction grades, the singular terms of the resummed expansion are consistent with a displaced simple pole whose formal location satisfies
\begin{equation}
 x_\pm=1+\alpha_1\epsilon^2+\alpha_2\epsilon^4+\cdots . \label{eq:corrected-pole-equation}
\end{equation}
For the $x_+$ family, the coefficients are
\begin{align}
 \alpha_1 &=-\frac{\gamma_{1,1}}{\gamma_{0,0}} =\frac{pd_1}{a}\gamma_{1,1},\quad\alpha_2=-\frac{\gamma_{1,2}}{\gamma_{0,0}} +\frac{2\gamma_{0,1}\gamma_{1,1} +\gamma_{1,1}^2}{2\gamma_{0,0}^2}. \label{eq:pole-displacements}
\end{align}
Applying $\mathcal R$ to these formulas gives the coefficients for the $x_-$ family. More precisely, the positive-tail part of \eqref{eq:resummed-grades} admits the formal identity, coefficientwise through order $\epsilon^4$,
\begin{equation}
 H_0+\epsilon^2H_1+\epsilon^4H_2+\cdots=\frac{\gamma_{0,0} +\epsilon^2(\gamma_{0,1}-\gamma_{1,1}) +\epsilon^4(\gamma_{0,2}-\gamma_{1,2}+\gamma_{2,2})+\cdots} {1-x_++\alpha_1\epsilon^2+\alpha_2\epsilon^4+\cdots}+B, \label{eq:displaced-pole-profile}
\end{equation}
where the coefficients of $B$ at these orders are holomorphic at $x_+=1$.
\end{proposition}

\begin{proof}
Expanding \eqref{eq:resummed-second-grade} at $x_+=1$ gives
\begin{align*}
 H_0&=\frac{\gamma_{0,0}}{1-x_+}+\mathcal O(1),\qquad H_1=\frac{\gamma_{1,1}}{(1-x_+)^2} +\frac{\gamma_{0,1}-\gamma_{1,1}}{1-x_+}+\mathcal O(1),\\
 H_2&=\frac{2\gamma_{2,2}}{(1-x_+)^3} +\frac{\gamma_{1,2}-3\gamma_{2,2}}{(1-x_+)^2} +\frac{\gamma_{0,2}-\gamma_{1,2} +\gamma_{2,2}}{1-x_+}+\mathcal O(1).
\end{align*}
Here, the $\mathcal {O} (1)$ terms are holomorphic at $x_+=1$. To compare with a displaced simple pole, expand its reciprocal denominator formally as
\begin{align*}
 \frac{1}{1-x_++\alpha_1\epsilon^2+\alpha_2\epsilon^4} &=\frac{1}{1-x_+}-\frac{\alpha_1\epsilon^2}{(1-x_+)^2}-\frac{\alpha_2\epsilon^4}{(1-x_+)^2}+\frac{\alpha_1^2\epsilon^4}{(1-x_+)^3}+\cdots.
\end{align*}
The numerator in \eqref{eq:displaced-pole-profile} reproduces the coefficients of $1/(1-x_+)$ at all three grades. Matching the $\epsilon^2/(1-x_+)^2$ term gives $-\gamma_{0,0}\alpha_1=\gamma_{1,1}$. At the next grade, matching $\epsilon^4/(1-x_+)^3$ requires
\begin{equation}
 2\gamma_{2,2} =\gamma_{1,1}^2/\gamma_{0,0}. \label{eq:simple-pole-compatibility}
\end{equation}
This is the displayed formula for $\gamma_{2,2}$, since $\gamma_{0,0}=-a/(pd_1)$.  Finally, the coefficient of $\epsilon^4/(1-x_+)^2$ gives
\[
 -\gamma_{0,0}\alpha_2 -(\gamma_{0,1}-\gamma_{1,1})\alpha_1 =\gamma_{1,2}-3\gamma_{2,2}.
\]
Using \eqref{eq:simple-pole-compatibility} yields $\alpha_2$ in \eqref{eq:pole-displacements}. The zero mode and the opposite tail are regular here and do not change the principal parts. Reflection gives the negative case.
\end{proof}

Equation~\eqref{eq:corrected-pole-equation} is the formal condition for the displaced denominator to vanish. After choosing a phase realisation and setting $\epsilon=1/t$, it gives an implicit equation for candidate pole locations in the $t$-plane. The terms $\alpha_1/t^2$ and $\alpha_2/t^4$ correct the singular phase value $x_\pm=1$; their conversion into displacements in $t$ requires local inversion of $x_\pm(t)$.

\section{One-sided pole arrays}
\label{sec:meromorphic-realisation}

Proposition~\ref{prop:one-sided-poles} establishes a local array of simple poles for a normalised, one-sided ($\kappa=0$) solution under explicit background and domain hypotheses. We fix an $r$-theta function realisation of the exponential factor \eqref{eq:QExpFactor}, and normalise the exponential perturbation relative to the exact $k=0$ quicksilver solution from Theorem~\ref{thm:one-sided-existence}. The corresponding formal limit in Section~\ref{sec:formal-bimodal} is $k_1=0$, so $\kappa=0$ and $\sigma=1$. Section~\ref{subsec:theta-boundary-approach} describes the annular anti-Stokes-like regions under this $r$-theta interpretation, and considers the formal deformation of the same pole array around $\kappa=0$.

For $\sigma\ne0$, let $r=1/q$ and introduce the $r$-theta function \cite[\S11.2]{GasperRahman2004},
\[
 \theta_r(w)=\prod_{j\ge0}(1-r^jw)(1-r^{j+1}/w),\qquad w\in\Cstar.
\]
We realise the formal exponential mode $Z$ by the product
\begin{align}\label{eq:rThetaZ}
 Z(t)=\theta_r(\lambda_1/t)\theta_r(\lambda_2/t),\qquad \lambda_1\lambda_2=r/(a\sigma),\qquad\lambda_1/\lambda_2\notin q^{\mathbb Z}.
\end{align}
This product is holomorphic on $\Cstar$ and, since $\theta_r(rw)=-\theta_r(w)/w$, satisfies
\begin{equation}
 Z(qt)=qa\sigma t^2Z(t),\qquad t\in\mathbb{C}^*. \label{eq:global-phase-equation}
\end{equation}
Substituting $X_+=Z(t)/t^3$ and $X_-=1/(t^3Z(t))$ makes each rational harmonic profile in Section~\ref{subsec:harmonic-summation} meromorphic on $\Cstar$. For the one-sided analytic family constructed below, we write $f(t;k)$ and $u(t;k)=-tf(t;k)$; the parameter $k$ argument is suppressed when its value is not at issue.

\begin{remark}
For fixed $a,q,\sigma$, any other nonzero meromorphic solution $\tilde Z$ of \eqref{eq:global-phase-equation} satisfies $\tilde Z=\phi Z$, where $\phi(qt)=\phi(t)$.  Such meromorphic functions $\phi(t)$ are called $q$-elliptic and form the constant field associated with the first-order $q$-difference equation \eqref{eq:global-phase-equation} \cite{Sauloy2006}. Inclusion of an arbitrary $\phi(t)$ such that $\tilde Z(t)$ remains holomorphic on $\mathbb{C}^*$ is equivalent to the freedom of choosing our amplitude $k$ along with the choice of $\lambda_1$. We leave $\lambda_1\in\mathbb{C}^*$ as a fixed, arbitrary constant, although we assume the generic case $\lambda_1/\lambda_2\notin q^{\mathbb Z}$.
\end{remark}

\subsection{Main result and proof outline}
\label{subsec:one-sided-pole-result}

Choose $\tau_0\in\lambda_1q^{\mathbb Z}\cup\lambda_2q^{\mathbb Z}$ and let $\tau_n=q^n\tau_0$, so each $\tau_n$ is a simple zero of $Z$. In the one-sided specialisation $\sigma=1$, $k_1=0$, $k_2=k$, write $x(t)=k/(apt^3Z(t))$, as prescribed by $x_-$ in Section~\ref{subsec:harmonic-summation}. We study neighbourhoods approaching $\tau_n$ as $n\to\infty$ on which $x(t)$ remains near $1$; introducing the local scale
\begin{equation}
 \omega_n=k/(ap\tau_n^3Z'(\tau_n)), \label{eq:one-sided-phase}
\end{equation}
this being the residue of $x$ at $\tau_n$:
\begin{align*}
 x(t)=\frac{\omega_n}{t-\tau_n}+\dots,\qquad t\to\tau_n.
\end{align*}

\begin{proposition}[One-sided pole array]
\label{prop:one-sided-poles}
Fix $a\ne0$, $|q|>1$, $k\ne0$, and an $r$-theta phase $Z$ according to \eqref{eq:rThetaZ} with $\sigma=1$. Let $U\Subset\Cstar$ be a sufficiently small base neighbourhood of $\tau_0$ as in Section~\ref{sec:quicksilver}. Fix the reference solution $f(s;0)$ supplied by Theorem~\ref{thm:one-sided-existence}, holomorphic on $\Omega_N$ and satisfying
\begin{equation}
 u(s;0)=1+1/(as^2)+\mathcal O(1/s^{4}), \label{eq:exact-one-sided-background}
\end{equation}
uniformly as $|s|\to\infty$ in $\Omega_N$. Fix $0<\delta<1/4$ and let $\mathcal D=\{\xi\in\mathbb C:|\xi-1|<\delta\}$.

Then, after increasing $N$, there is a holomorphic function $f(s;k)$ on
\[
 \{s\in\Omega_N:Z(s)\ne0,\ |k/(sZ(s))|<2\},
\]
which satisfies \eqref{eq:qPII} as a meromorphic identity wherever all three arguments $s/q,s,qs$ belong to this domain, with normalisation uniform in this domain:
\begin{equation}
 f(s;k)-f(s;0)=\frac{k}{s^2Z(s)} \{1+\mathcal O(1/s^{2})\}\quad\text{as}\quad|s|\to\infty. \label{eq:theta-family-normalisation}
\end{equation}

For all sufficiently large $n$, $x$ has a holomorphic inverse branch $\psi_n$ defined on a neighbourhood of $\overline{\mathcal D}$, with $x(\psi_n(\xi))=\xi$ and $(\psi_n(\xi)-\tau_n)/\omega_n\to 1/\xi$ as $n\to+\infty$ uniformly on $\overline{\mathcal D}$. The solution $f(s;k)$ continues meromorphically backwards by \eqref{eq:qPII} to $\psi_n(\mathcal D)$, where it has exactly one pole $t_n^\ast$. This pole is simple, and as $n\to\infty$ its phase satisfies
\begin{equation}
 x(t_n^\ast) =1+\alpha_1/(t_n^\ast)^2+\mathcal O(1/\tau_n^{4}),\qquad \alpha_1=(p^2+2-ap)/(ap^2(p-1)). \label{eq:true-first-pole-displacement}
\end{equation}
Equivalently, its location in the $t$-plane is
\begin{equation}
 t_n^\ast=\tau_n+\omega_n \{1-\alpha_1/\tau_n^2+\mathcal O(1/\tau_n^{4})\}. \label{eq:true-pole-theta-location}
\end{equation}
\end{proposition}

The formal expansion in Section~\ref{subsec:harmonic-summation} gives the leading one-sided profile
\begin{equation}
 \mathcal RH_0=apx/(1-x),\qquad u(t)=1-ap\,t^2x/(1-x)+\cdots , \label{eq:one-sided-leading-profile}
\end{equation}
predicting the singular phase $x=1$. To establish these predictions for the normalised solution, we need uniform estimates on the sequential neighbourhoods where $x$ is close to $1$. The proof has three steps.

\begin{enumerate}
\item Obtain a uniform approximation to $u(qt;k)$, including its first algebraic correction, for bounded $x(t)$. This uses the exact background from Theorem~\ref{thm:one-sided-existence} and is established in Lemma~\ref{lem:uniform-one-sided}, proved in Appendix~\ref{app:one-sided-estimates}.

\item Construct local inverse branches of $x$ on $\mathcal D$ to identify the corresponding neighbourhoods near the $r$-theta zeros $\tau_n$. Lemma~\ref{lem:theta-inverse-coordinates} gives the scale $\omega_n$ needed to translate phase estimates into $t$-plane locations.

\item Use the uniform approximation and Rouch\'e's theorem to locate one simple zero of $u(qt;k)$ in each neighbourhood. The recurrence converts this zero into a simple pole of $u(t;k)$ (Lemma~\ref{lem:forward-pole-criterion}). The first correction and the local inverse then determine the pole's location.
\end{enumerate}
This implements the same one-sided transasymptotic strategy as the classical Painlev\'e pole calculations of Costin and Costin \cite[\S2.3 and \S\S5.2--5.3]{CostinCostin2001}: nonnegative harmonics of one recessive exponential are summed into coefficient profiles, and uniform asymptotic control relates their singularities to poles of true solutions. See also \cite{OldeDaalhuis2022}.  Here, the uniform $q$-difference estimate and the exact recurrence provide the passage from profiles to poles. The formal bimodal expansion motivates the result and supplies a consistency check; the analytic proof derives the required coefficients from \eqref{eq:qPII} and \eqref{eq:exact-one-sided-background}, without assuming the conjectured all-harmonic formulas of Section \ref{sec:formal-bimodal}.

\subsection{Proof of Proposition \ref{prop:one-sided-poles}}
\label{subsec:one-sided-profile}
\label{subsec:pole-transfer}

We use Rouch\'e's theorem in the following form: if $h_1,h_2$ are holomorphic near a closed disc and $|h_1-h_2|<|h_2|$ on its boundary, then they have the same number of zeros inside, counted with multiplicity. In particular, comparison with a function having one simple zero proves both existence and simplicity of the perturbed zero.

For $u(s)=-sf(s)$, the transformed equation \eqref{eq:scaled-bimodal-equation} reads
\begin{equation}
 as^4(u(s)-1)=u(qs)u(s/q)u(s)(s^2-u(s)), \label{eq:normalised-u-equation}
\end{equation}
The phase coordinate $x$ satisfies, by \eqref{eq:global-phase-equation},
\begin{equation}
 x(qt)=x(t)/(ap^2t^2)\qquad\text{where}\qquad x(s/q)=k/(sZ(s)). \label{eq:normalised-recessive-amplitude}
\end{equation}
For uniform estimates, we use an independent variable $\xi$ in place of $x(t)$. The function $\mathcal U(t,\xi)$ below represents $u(qt;k)$ with variable $\xi=x(t)$ rather than $k$.

\begin{lemma}[Uniform one-sided approximation]
\label{lem:uniform-one-sided}
Use the background and domains of Proposition~\ref{prop:one-sided-poles}. For every fixed real $\mu>0$, after increasing $N$, there is a holomorphic function $\mathcal U(t,\xi)$ on $\Omega_N/q\times\{|\xi|<2\mu\}$, with $\mathcal U(t,0)=u(qt;0)$ and, uniformly for $qt\in\Omega_N$, $|\xi|\le \mu$,
\begin{equation}
 \mathcal U(t,\xi) =1-\xi+A(\xi)/t^2+\mathcal O(1/t^{4}), \label{eq:uniform-amplitude-expansion}
\end{equation}
where
\begin{equation}
 A(\xi)=\frac1{ap}-\frac{\xi^2}{ap^2} -\frac{ap-2p-1}{ap^2(p-1)}\xi. \label{eq:actual-first-correction}
\end{equation}
For $qt\in\Omega_N$ and $|\xi|<2\mu$, let $s_j=q^{j+1}t$ and $\xi_j=\xi a^{-j}t^{-2j}q^{-j(j+3)}$.  Then $u_j=\mathcal U(s_j/q,\xi_j)$ solves \eqref{eq:normalised-u-equation} for $j\ge1$, with the normalisation
\[
 u_j-u(s_j;0)=-\xi_j\{1+\mathcal O(1/s_j^{2})\}.
\]
The constructions agree under forward shifts and on common domains for different $\mu$. Consequently $u(s;k)=\mathcal U(s/q,x(s/q))$ and $f(s;k)=-u(s;k)/s$ are holomorphic where $s\in\Omega_N$, $Z(s)\ne0$ and $|x(s/q)|<\mu$.  The function $f(s;k)$ satisfies \eqref{eq:qPII} as a meromorphic identity wherever the three shifted arguments belong to this domain. On a sufficiently late tail of each fixed orbit avoiding the zeros of $Z$, $f(s;k)$ satisfies the summation equation \eqref{eq:one-sided-fixed-point} for a constant chosen on that orbit, and hence belongs to the family of Theorem~\ref{thm:one-sided-existence}.
\end{lemma}

The proof is given in Appendix~\ref{app:one-sided-estimates}.

\begin{lemma}[Local inversion of the phase coordinate]
\label{lem:theta-inverse-coordinates}
For all sufficiently large $n$, the map $x(t)$ has a holomorphic one-to-one inverse $t=\psi_n(\xi)$ on a neighbourhood of $\overline{\mathcal D}$. Uniformly there,
\begin{align}
 \psi_n(\xi)&=\tau_n+\omega_n/\xi +\mathcal O\left(n\omega_n^2/\tau_n\right),\quad\psi_n'(\xi)=-\omega_n/\xi^2 +\mathcal O\left(n\omega_n^2/\tau_n\right). \label{eq:theta-inverse-estimates}
\end{align}
In particular $|\psi_n(\xi)|\asymp |\tau_n|$, and $|\omega_n|$ decays faster than every fixed power of $1/|\tau_n|$.
\end{lemma}

\begin{proof}
Iteration of \eqref{eq:global-phase-equation}, with $\sigma=1$, gives
\begin{equation}
 Z(q^ns)=(as^2)^nq^{n^2}Z(s),\qquad Z'(\tau_n)=(a\tau_0^2)^nq^{n^2-n}Z'(\tau_0). \label{eq:one-sided-theta-orbit}
\end{equation}
The second identity follows by differentiation at $s=\tau_0$ and proves the stated decay of $\omega_n$.

Introduce a complex local coordinate $v$ by $t=\tau_n+\omega_nv$, and restrict it to $|v-1|\le1/2$.  Then $s=t/q^n=\tau_0+\omega_nv/q^n$ lies near $\tau_0$ for large $n$. Taylor expansion in \eqref{eq:one-sided-theta-orbit} gives
\[
 Z(\tau_n+\omega_nv) =Z'(\tau_n)\omega_nv\left\{1+ \mathcal O\left(n\omega_n/\tau_n\right)\right\}
\]
uniformly on this disc. Since $|v|\ge1/2$, the displayed expression is nonzero for large $n$, so $v\mapsto x(\tau_n+\omega_nv)$ is holomorphic near the closed disc. By \eqref{eq:one-sided-phase} and the definition of $\omega_n$,
\[
 x(\tau_n+\omega_nv)=1/v+ \mathcal O\left(n\omega_n/\tau_n\right).
\]

For each $\xi$ with $|\xi-1|<1/4$, we solve $x(\tau_n+\omega_nv)=\xi$ in $|v-1|<1/2$.  On the boundary,
\[
 \left|1/v-\xi\right| \ge|1/v-1|-|\xi-1| \ge1/3-1/4=1/12.
\]
The error above is smaller than $1/12$ for all sufficiently large $n$, independently of $\xi$. Rouch\'e's theorem therefore gives the same zero count as $1/v-\xi$, whose unique zero $1/\xi$ lies in the disc since $|1/\xi-1|<1/3$. Thus there is exactly one zero, counted with multiplicity; it is simple and depends holomorphically on $\xi$ by the implicit-function theorem. At this zero, the displayed expansion gives $v=1/\xi+\mathcal O(n\omega_n/\tau_n)$.

Define $\psi_n(\xi)=\tau_n+\omega_nv$ using this unique zero.  This proves the first estimate in \eqref{eq:theta-inverse-estimates}, uniformly on $|\xi-1|<1/4$.  Moreover, $x(\psi_n(\xi))=\xi$ makes $\psi_n$ one-to-one.  Since $\delta<1/4$, Cauchy's estimate for the holomorphic error on a smaller fixed disc still containing $\overline{\mathcal D}$ gives the derivative estimate. Finally, the first estimate and $\omega_n/\tau_n\to0$ imply $|\psi_n(\xi)|\asymp|\tau_n|$ there.
\end{proof}

\begin{lemma}[Forward estimates and pole criterion]
\label{lem:forward-pole-criterion}
For the family in Lemma~\ref{lem:uniform-one-sided}, let $V(t)=u(qt)$ and $W(t)=u(q^2t)$. For every fixed $\mu>0$, take $N$ sufficiently large in Lemma~\ref{lem:uniform-one-sided}.  Then $V$ and $W$ are holomorphic near every point with $qt\in\Omega_N$, $Z(t)\ne0$ and $|x(t)|\le \mu$, and uniformly on this set as $|t|\to\infty$,
\begin{align}
 V(t)&=1-x(t)+A(x(t))/t^{2}+\mathcal O(1/t^{4}), \label{eq:shifted-first-corrected-estimate}\\
 W(t)&=1+\frac{1-x(t)}{ap^2t^2}+\mathcal O(1/t^{4}), \label{eq:forward-estimate-target}
\end{align}
where $A$ is defined in \eqref{eq:actual-first-correction}. On a neighbourhood where $V$ is not identically zero and $W$, $1-V$ and $1-V/(pt^2)$ are nonzero, the identity
\begin{equation}
 \frac1{u(t)} =-\frac{W(t)V(t)\{1-V(t)/(pt^2)\}}{ap\,t^2\{1-V(t)\}} \label{eq:shifted-reciprocal-identity}
\end{equation}
defines a meromorphic backward continuation. Its poles are precisely the zeros of $V$, with the same multiplicities.
\end{lemma}

\begin{proof}
Since $V(t)=\mathcal U(t,x(t))$ and $W(t)=\mathcal U(qt,x(qt))$, \eqref{eq:uniform-amplitude-expansion} and $x(qt)=x(t)/(ap^2t^2)$ give the two estimates.  Solving \eqref{eq:normalised-u-equation} at $qt$ for the preceding value gives \eqref{eq:shifted-reciprocal-identity}. Under the stated conditions, its right-hand side is a nonzero holomorphic factor times $V$ and is not identically zero. Its reciprocal therefore defines a meromorphic function $u(t)$. On overlaps with the original domain, this agrees with the constructed solution, which already satisfies the same recurrence by shift compatibility. Since $f(t)=-u(t)/t$ and $t\ne0$, $f$ and $u$ have the same poles with the same multiplicities.
\end{proof}

We proceed to complete the proof of Proposition~\ref{prop:one-sided-poles}.

\begin{proof}[Proof of Proposition~\ref{prop:one-sided-poles}]
Lemma~\ref{lem:uniform-one-sided}, with $\mu=2$, supplies $f$ on the stated domain with normalisation \eqref{eq:theta-family-normalisation}. Let $\psi_n$ be the inverse branches from Lemma~\ref{lem:theta-inverse-coordinates}, and work on a fixed neighbourhood of $\overline{\mathcal D}$ contained in $|\xi-1|<1/4$. Since $\psi_n(\xi)/q^{n}\to\tau_0\in U$ uniformly there, $q\psi_n(\xi)$ and $q^2\psi_n(\xi)$ belong to $\Omega_N$ for all sufficiently large $n$.  Also $Z(\psi_n(\xi))\ne0$ and $|x(\psi_n(\xi))|=|\xi|<2$.  Lemma~\ref{lem:forward-pole-criterion} therefore applies: the functions $V_n(\xi)=V(\psi_n(\xi))$ and $W_n(\xi)=W(\psi_n(\xi))$ are holomorphic on this neighbourhood and satisfy, uniformly there,
\begin{equation}
 V_n(\xi)=1-\xi+A(\xi)/\psi_n(\xi)^{2}+\mathcal O(1/\tau_n^{4}),\quad W_n(\xi)=1+\mathcal O(1/\tau_n^{2}). \label{eq:shifted-uniform-target}
\end{equation}

On $\partial\mathcal D$, the difference $V_n(\xi)-(1-\xi)$ is $\mathcal O(1/\tau_n^{2})$, hence smaller than $|1-\xi|=\delta$ for large $n$.  Rouch\'e's theorem gives exactly one zero of $V_n$ in $\mathcal D$, counted with multiplicity. Denote this simple zero by $\xi_n^\ast$ and set $t_n^\ast=\psi_n(\xi_n^\ast)$. Meanwhile $W_n$, $1-V_n$ and $1-V_n/(p\psi_n^2)$ tend uniformly to $1$, $\xi$ and $1$, respectively, and hence do not vanish on $\overline{\mathcal D}$, since $|\xi|\ge1-\delta>0$ there.  The pole criterion in Lemma~\ref{lem:forward-pole-criterion} now shows that one backward step defines a meromorphic continuation with precisely one simple pole in $\psi_n(\mathcal D)$, at $t_n^\ast$.

To locate this pole more precisely, first use $V_n(\xi_n^\ast)=0$ in \eqref{eq:shifted-uniform-target} to obtain $\xi_n^\ast=1+\mathcal O(1/\tau_n^{2})$. Since $A$ is a fixed polynomial, this implies $A(\xi_n^\ast)=A(1)+\mathcal O(1/\tau_n^{2})$. Substituting back into the same equation gives
\[
 \xi_n^\ast=1+A(1)/(t_n^\ast)^{2}+\mathcal O(1/\tau_n^{4}).
\]
As $\xi_n^\ast=x(t_n^\ast)$ and $A(1)=\alpha_1$ by \eqref{eq:actual-first-correction}, this proves \eqref{eq:true-first-pole-displacement}.

Finally, \eqref{eq:theta-inverse-estimates} gives $t_n^\ast-\tau_n=\mathcal O(\omega_n)$, and therefore $1/(t_n^\ast)^2=1/\tau_n^2+\mathcal O\left(\omega_n/\tau_n^3\right)$. Insert this and $1/\xi_n^\ast=1-\alpha_1/(t_n^\ast)^2+\mathcal O(1/\tau_n^{4})$ into the first estimate of \eqref{eq:theta-inverse-estimates}. The decay of $\omega_n$ from Lemma~\ref{lem:theta-inverse-coordinates} absorbs both the inverse-coordinate error and the substitution error into $\mathcal O(\omega_n/\tau_n^{4})$, therefore we obtain \eqref{eq:true-pole-theta-location}.
\end{proof}

\subsection{Annular anti-Stokes-like regions}
\label{subsec:theta-boundary-approach}

We consider the implication of the meromorphic $r$-theta realisation \eqref{eq:rThetaZ} in the broader context of the expansions of Section \ref{sec:formal-bimodal}. Let $k_1k_2\ne0$ and fix the selected $\sigma(\kappa)$. The two leading geometric sums of Section~\ref{subsec:harmonic-summation} converge simultaneously in
\begin{equation}
 |x_-|<1,\quad |x_+|<1 \quad\Longleftrightarrow\quad |k_2d_1|/|apt^3|<|Z(t)| <|at^3|/|pk_1d_1|. \label{eq:theta-harmonic-band}
\end{equation}
Throughout this region, $\log|Z(t)|=\mathcal O(\log|t|)$, so the two modes $Z(t)$ and $1/Z(t)$ (alternatively $X_+,X_-$) differ by at most algebraic factors in $|t|$.  We therefore interpret such a local annular component as an \emph{anti-Stokes transition region}. Edges given by $|x_-|=1$ and $|x_+|=1$ are the convergence boundaries of the respective harmonic sums, given by algebraically weighted level curves of the product of $r$-theta functions defining $Z$.

Recall that we assume the zero spirals $\lambda_1q^{\mathbb Z}$ and $\lambda_2q^{\mathbb Z}$ are distinct, so their zeros are simple, and consider the zeros along either spiral at $\tau_n=q^n\tau_0$. The local Taylor expansion used in Lemma~\ref{lem:theta-inverse-coordinates} shows that the inner and outer edges of the proposed annular regions have radii asymptotic to
\begin{equation}
 \frac{|k_2d_1|}{|\tau_n|^3|apZ'(\tau_n)|}, \qquad \frac{|a|\,|\tau_n|^3}{|pk_1d_1Z'(\tau_n)|}, \label{eq:theta-level-radii}
\end{equation}
respectively. Iterating \eqref{eq:global-phase-equation} gives $|Z'(\tau_n)|=|q|^{n^2+\mathcal O(n)}$:  for fixed parameters, as $n\to\infty$, each edge translated by $-\tau_n$ and rescaled by its corresponding radius tends to the unit circle. The annuli shrink in absolute size around centres lying exactly on the two $q$-spirals. The formal pole calculation \eqref{eq:corrected-pole-equation} places the pole families associated with the negative and positive harmonics asymptotically on the inner and outer edges, respectively.

In the one-sided limit $k_1=0$, $k_2=k$, only the convergence boundary $|x|=1$ remains. Proposition~\ref{prop:one-sided-poles} shows that the true poles approach this boundary relative to its shrinking local radius. The analogue of a classical anti-Stokes ray is thus a family of disjoint boundary curves around $r$-theta zeros on $q$-spirals.

\begin{remark}
Introducing the second exponential mode formally perturbs the one-sided pole array. To compare locations as $\kappa$ varies from zero, fix $k_2=k$ and choose an $r$-theta family $Z_\kappa$, with $Z_0=Z$, keeping both the zero $\tau_0$ and the derivative $Z_\kappa'(\tau_0)$ fixed; the latter can be achieved by a scalar normalisation. The $r$-theta shift law and the formal pole condition \eqref{eq:corrected-pole-equation} then give
\[
 t^*_n(\kappa) =\tau_n-d_1(\kappa)\sigma(\kappa)^{-n}\omega_n \left\{1-\alpha_1(\kappa)/\tau_n^2+\cdots\right\},
\]
where $\tau_n,\omega_n$ are the one-sided scales and $\alpha_1(\kappa)$ is the coefficient for the $x_-$ family in \eqref{eq:pole-displacements}. At $\kappa=0$, this agrees through the first correction with \eqref{eq:true-pole-theta-location}.
\end{remark}

\section*{Summary}

We established local one-parameter families of quicksilver solutions with a common algebraic asymptotic expansion. A formal bimodal expansion investigates the interaction of two independent exponential modes. Summing over harmonics yields rational profiles that indicate the locations of simple poles.

For a normalised one-sided family with a fixed $r$-theta phase, we proved the existence of a simple-pole array and obtained its first displacement correction. The poles approach the associated harmonic convergence boundaries, whose curves are organised along $q$-spirals. This gives a $q$-difference analogue of classical anti-Stokes pole geometry.

\printbibliography

\appendix

\section{Calculated Formal Coefficients}
\label{app:Coefficients}

We display further coefficients of the formal expansion \eqref{eq:u-bimodal-series}, according to the proposed form of Conjecture \ref{conj:polynomial-corrections}.

We have already shown $M_0=1$ and $\gamma_{0,0}=-a/(pd_1)$, by \eqref{eq:leading-coefficients}. At grade $j=1$, the proposed tail has $M_1=2$ and coefficients
\begin{equation}
 \gamma_{0,1}=(p+1)d_1,\qquad \gamma_{1,1}=(ap+1+2d_1(p^2+1))/(1-p). \label{eq:first-harmonic-correction}
\end{equation}
At grade $j=2$, the proposed tail has $M_2=3$ and coefficients
\begin{equation}
\begin{aligned}
 \gamma_{0,2} &=\frac{d_1}{a}\{pa+2d_1(p^3a+p^2+pa+1) +3d_1^2(p^4+p^2+1)\},\\
 \gamma_{1,2}&=\frac{d_1}{2a(p^2-1)}\{ p^3a^2-4p^3a-4pa+p+d_1(-6p^5a+2p^4a-6p^4-2p^3a+2p^3\\
 &\hspace{6em}+2p^2a-2p^2-6pa+2p-6)+d_1^2(-8p^6+2p^5+8p^3+2p-8)\},\\
 \gamma_{2,2}&=-\frac{pd_1}{2a}\gamma_{1,1}^2.
\end{aligned}
 \label{eq:second-harmonic-correction}
\end{equation}

Meanwhile, the next two zero-harmonic coefficients are given by
\begin{align*}
 d_2&=\frac{d_1^2}{a}(2+3d_1),\quad d_3=\frac{d_1^3}{a^2}\left\{3a+4d_1(a+1+p+1/p) +5d_1^2(1+p+1/p)\right\},
\end{align*}
and the positive-harmonic boundary coefficients ($1\le m<M_j$) for $j=1,2$ are
\begin{align*}
 c_{1,1}&=\frac{k_1pd_1}{a(p-1)}\{a+1+d_1(p+3)\},\\
 c_{1,2}&=\frac{k_1pd_1^2}{a^2(p-1)^2(p+1)}\{ 5p^2a+pa^2-3pa+p\\
 &\qquad+d_1(p^4a+5p^3+12p^2a-5pa+8p-5)+d_1^2(p^5+11p^3+p^2+10p-7)\}.
\end{align*}
We omit the lengthy expression of $c_{2,2}$.

\section{Proof of Lemma~\ref{lem:uniform-one-sided}}
\label{app:one-sided-estimates}

\begin{proof}
Fix $\mu>0$, let $s=qt\in\Omega_N$ and $|\xi|\le2\mu$. For $s_j=q^js$, let $\xi_0=\xi$ with $\xi_j=\xi_{j-1}/(as_j^2)$, $j\ge1$, i.e., corresponding to the $q$-difference equation satisfied by $x(t)$. We introduce coordinate $w_j$ by $u_j=u(s_j;0)-\xi_jw_j$ and seek to show that $w_j\to1$ as $j\to\infty$, recalling the reference solution $u(s_j;0)$ satisfying \eqref{eq:exact-one-sided-background}.

Here, $w=(w_j)_{j\ge0}$ is constructed in $l^\infty(\mathbb N)$, the Banach space of bounded complex sequences indexed by $\mathbb N=\{0,1,2,\ldots\}$, with norm $\|w\|_\infty=\sup_{j\ge0}|w_j|$. We proceed to work in the closed ball $\|w-\mathbf1\|_\infty\le1$, where $\mathbf1=(1,1,\ldots)$; in particular, every sequence in this ball satisfies $|w_j|\le2$.

For $\xi\ne0$, subtracting \eqref{eq:normalised-u-equation} evaluated at the exact background $u(s_j;0)$ from the same equation for the trial values $u_j$, dividing by $-as_j^4\xi_j$, and rearranging shows that $u_j$ satisfies \eqref{eq:normalised-u-equation} for $j\ge1$ if and only if
\begin{equation}
 w_j-w_{j-1}=\mathcal J_j(w),\qquad j\ge1, \label{eq:normalised-tail-increment}
\end{equation}
where
\begin{align}
 \mathcal J_j(w):=& \left\{u_{j+1}u_j(1-u_j/s_j^2)-1\right\}w_{j-1}+\frac{1}{a^2s_j^2s_{j+1}^2}u(s_{j-1};0)u_j(1-u_j/s_j^2) w_{j+1}\notag\\
 &+\frac{1}{as_j^2}u(s_{j+1};0)u(s_{j-1};0) \left\{1-(u(s_j;0)+u_j)/s_j^2\right\}w_j. \label{eq:normalised-tail-polynomial}
\end{align}
The right-hand side is a polynomial in the three input components $w_{j-1},w_j,w_{j+1}$, and is holomorphic at $\xi=0$.

For sufficiently large $N$, $1/|as_j^2|\le1$, so that $|\xi_j|\le 2\mu/|as_j^2|$ for $j\ge1$. Together with \eqref{eq:exact-one-sided-background}, this gives $u_j=1+\mathcal O(1/s_j^{2})$ for $j\ge1$, uniformly over all trial sequences $w$ in the ball; therefore, \eqref{eq:normalised-tail-polynomial} shows that $\mathcal J_j(w)=\mathcal{O}(1/s_j^2)$.  Derivatives with respect to the input components have the same bounds. Consequently, for any two sequences $w,\widetilde w$ in the ball and every $j\ge1$,
\[
 |\mathcal J_j(w)|\le C/|s_j|^{2},\qquad |\mathcal J_j(w)-\mathcal J_j(\widetilde w)| \le (C/|s_j|^{2})\|w-\widetilde w\|_\infty,
\]
where $C$ may depend on $\mu$ and the fixed data, but not on $s,\xi$ or the shift index in the stated ranges.

For each fixed $s,\xi$, define a map $\mathcal S$ on this ball by
\begin{equation}
 (\mathcal S w)_n =1-\sum_{j=n+1}^{\infty}\mathcal J_j(w),\qquad n\ge0. \label{eq:normalised-tail-operator}
\end{equation}
This defines a new sequence for an arbitrary input $w$; the equation to solve is $w=\mathcal S w$. Indeed, summing \eqref{eq:normalised-tail-increment} with $w_j\to1$ gives precisely this fixed-point equation. Conversely, subtracting adjacent component equations recovers \eqref{eq:normalised-tail-increment}, while the tail bound gives $w_n\to1$.

Since $\sum_{j>n}|1/s_j|^{2}=|s_n|^{-2}/(|q|^2-1)$, the series in \eqref{eq:normalised-tail-operator} converge uniformly.  For all $w,\widetilde w$ in the ball,
\begin{align*}
 \|\mathcal S w-\mathbf1\|_\infty &\le\frac{C}{(|q|^2-1)|s|^{2}},\qquad\|\mathcal S w-\mathcal S\widetilde w\|_\infty\le\frac{C\|w-\widetilde w\|_\infty}{(|q|^2-1)|s|^{2}}.
\end{align*}
Increasing $N$ makes the displayed factor at most $1/2$, uniformly in $s,\xi$.  Thus $\mathcal S$ maps the ball into itself and is a contraction. The \emph{Picard iterates} are the successive approximations
\[
 w^{(0)}=\mathbf1,\qquad w^{(m+1)}=\mathcal S w^{(m)},\qquad m=0,1,2,\ldots.
\]
For each fixed $(s,\xi)$, the contraction mapping theorem guarantees that the Picard iterates $w^{(m)}$ converge in the supremum norm to the unique fixed point $w=w(s,\xi)$ in the ball. Its defining equation also yields $w_n=1+\mathcal {O} (1/s_n^{2})$.

Now define $\mathcal U(t,\xi)=u(s;0)-\xi w_0(s,\xi)$.  Each component of each Picard iterate is holomorphic in $s,\xi$, by locally uniform convergence of the defining sums.  The uniform contraction bound makes their limit locally uniform, so $\mathcal U$ is holomorphic on $\Omega_N/q\times\{|\xi|<2\mu\}$ and $\mathcal U(t,0)=u(qt;0)$. Uniqueness also makes the constructions for different $\mu$ agree wherever both are defined. Starting the construction at $qs,\xi_1$ yields precisely the fixed-point problem for the tail $(w_1,w_2,\ldots)$; uniqueness, therefore, implies shift compatibility. Reversing the subtraction that led to \eqref{eq:normalised-tail-increment} shows that $u_j=\mathcal U(s_j/q,\xi_j)$ solves \eqref{eq:normalised-u-equation} for $j\ge1$ when $\xi\ne0$, and hence also at $\xi=0$ by holomorphy. Substitution of $\xi=x(s/q)$ from \eqref{eq:normalised-recessive-amplitude} gives $u(s;k)=\mathcal U(s/q,x(s/q))$ and $f(s;k)=-u(s;k)/s$ on the domain in the lemma. That domain is forward invariant for large $N$ by \eqref{eq:normalised-recessive-amplitude}. If $s/q,s,qs$ all belong to it, the sequence construction with initial point $s/q$ proves the equation at $s$.

For the first correction, compare the fixed point with the first Picard iterate $w^{(1)}=\mathcal S\mathbf1$. The contraction estimate gives
\[
 \|w-w^{(1)}\|_\infty \le\frac{C\|w-\mathbf1\|_\infty}{(|q|^2-1)|s|^{2}}=\mathcal O(1/s^{4}).
\]
Using \eqref{eq:exact-one-sided-background}, $s_{j+1}=qs_j$, $\xi_j=\mathcal O(1/s_j^{2})$ and $\xi_{j+1}=\mathcal O(1/s_j^{4})$ in \eqref{eq:normalised-tail-polynomial} gives, at $w=\mathbf1$ and uniformly for $j\ge1$,
\[
 \mathcal J_j(\mathbf1) =\left(2/a+1/(ap)-1\right)/s_j^{2}-\xi_j +\mathcal O(1/s_j^{4}).
\]
The recurrence for $\xi_j$ and the geometric series give
\begin{align*}
 \sum_{j\ge1}1/s_j^{2}&=\frac{1}{(p-1)s^{2}},\quad \sum_{j\ge1}|1/s_j|^{4}=\mathcal O(1/s^{4}),\quad\sum_{j\ge1}\xi_j=\frac{\xi}{aps^{2}}+\mathcal O(1/s^{4}).
\end{align*}
Evaluating \eqref{eq:normalised-tail-operator} at $n=0$, $w=\mathbf1$, and using the bound for $w-w^{(1)}$ yields
\[
 w_0=1+\left\{\frac{ap-2p-1}{ap(p-1)} +\frac{\xi}{ap}\right\}\frac{1}{s^{2}} +\mathcal O(1/s^{4}),
\]
and hence \eqref{eq:uniform-amplitude-expansion}, since $\mathcal U(t,\xi)=u(qt;0)-\xi w_0(qt,\xi)$ and $\xi$ is bounded.

Finally, $f(s;k)-f(s;0)=x(s/q)w_0(s,x(s/q))/s$ gives \eqref{eq:theta-family-normalisation}.  On any fixed orbit based at $t_0\in U$ with $Z(t_0)\ne0$, the $r$-theta shift law \eqref{eq:global-phase-equation} makes $x(q^{n-1}t_0)$ decay faster than every algebraic power of $1/|q^nt_0|$.  This decay is uniform on compact subsets of the base that avoid the zeros of $Z$. To identify the orbitwise solution with Theorem~\ref{thm:one-sided-existence}, let $v_n=f(t_n;k)-D_n$ in the notation of Section~\ref{sec:quicksilver}.  Since $f(\cdot;0)-D$ and $f(\cdot;k)-f(\cdot;0)$ are algebraically flat, so is $v$. Lemma~\ref{lem:nonlinear-estimates} and Lemma~\ref{lem:summation-bound} then imply that $\mathcal K\mathcal Nv$ is algebraically flat, and by \eqref{eq:right-inverse}, $\mathcal A(v-\mathcal K\mathcal Nv)=0$. The dichotomy \eqref{eq:Pminus}--\eqref{eq:Pplus} therefore gives $v-\mathcal K\mathcal Nv$ as a constant multiple of $P^-$. Fix a radius $\rho>0$. As the starting index $N$ increases, $\|v\|_4+\|\mathcal K\mathcal Nv\|_4\to0$, by flatness and the uniform bound in Lemma~\ref{lem:summation-bound}.  Thus the homogeneous term $v-\mathcal K\mathcal Nv$ has norm less than $\rho/2$ on a sufficiently late tail. The estimates in the proof of Theorem~\ref{thm:one-sided-existence} then make the right-hand side of \eqref{eq:one-sided-fixed-point}, with this multiplier and at this same starting index, a contraction of the $X_4$ ball of radius $\rho$ into itself. The constructed $v$ lies in that ball and is its unique fixed point. The multiplier is chosen separately on each orbit and may depend on its base point; $k$ denotes the fixed $r$-theta amplitude in \eqref{eq:theta-family-normalisation}.
\end{proof}

\end{document}